\documentclass[11pt,a4paper,oneside,reqno]{amsart}

\usepackage[T1]{fontenc}
\usepackage{lmodern}
\usepackage{microtype}
\usepackage[a4paper,margin=27mm,headheight=14pt,headsep=7mm]{geometry}
\usepackage{comment}
\usepackage{xcolor}

\usepackage{amsmath,amsthm}
\usepackage{amssymb}

\usepackage{enumitem}
\usepackage{url}

\usepackage{caption}
\usepackage{array} % add to the preamble to make the table
\newcolumntype{M}[1]{>{\centering\arraybackslash}m{#1}} % add to the preamble to make the table
\usepackage{stackengine} % add to the preamble to make the table
\makeatletter
\newcommand{\customlabel}[2]{%
	\protected@write \@auxout {}{\string \newlabel {#1}{{#2}{}}}}
\makeatother

\newenvironment{proofbold}[1][Proof]{%
  \begin{proof}[#1:]\ignorespaces
}{%
  \end{proof}%
  \par\addvspace{0.6\baselineskip}%
}
\newcommand{\mbinom}[2]{\bigg(\genfrac{}{}{0pt}{}{#1}{#2}\bigg)}

\newtheorem{theorem}{Theorem}[section]
\newtheorem{corollary}[theorem]{Corollary}
\newtheorem{lemma}[theorem]{Lemma}

\theoremstyle{plain}

\newtheorem{remark}[theorem]{Remark}

\numberwithin{equation}{section}

\newtheorem*{theorem*}{Theorem}

\usepackage{dsfont}

\DeclareMathAlphabet{\mathpzc}{OT1}{pzc}{m}{it}

\begin{document}

\title[A recurrence in logarithmic M\"obius sums]{A recurrence in logarithmic M\"obius sums with coprimality conditions via Euler--Maclaurin}

\author{Sebastian Zuniga-Alterman}

\newcommand{\Addresses}{{
  {\footnotesize
\ \\
\noindent S.~ZUNIGA-ALTERMAN - \textsc{School of Science, The University of New South Wales, Canberra, Northcott Drive, Campbell ACT 2612, AUSTRALIA.}\\
  \texttt{s.zuniga$\_$alterman@unsw.edu.au}

}}}

\thanks{The author has been supported by the Australian Research Council Discovery Project DP240100186.}

\begin{abstract}
We derive a recurrence for logarithmically weighted Möbius sums by combining Möbius inversion with the Euler--Maclaurin summation formula. Consequently, starting from estimates of orders zero, one and two, we obtain polynomial approximations for these sums for every nonnegative integer exponent of $\log$. We then extend these estimates to sums with coprimality conditions via a convolution identity and to sums weighted by the reciprocal of the sum-of-divisors function. We also study the nonnegativity of the latter sums for every positive integer logarithmic order, establishing that they never change sign.
\end{abstract}

\maketitle

\normalsize

%%%%%%%%%%%%%%%%%%%%%%%%%%%%%%%%%%%%%%%%%%%%%%%%%%%%%%%%%%%%%%
\section{Introduction}
%%%%%%%%%%%%%%%%%%%%%%%%%%%%%%%%%%%%%%%%%%%%%%%%%%%%%%%%%%%%%%

For $X>0$, $k\in\mathbb{Z}_{\geq 0}$ and
$q\in\mathbb{Z}_{\geq 1}$, consider the logarithmic
M\"obius sums
\begin{align}\label{definitions}
\mathbf{\check m}_q^{(k)}(X)
=\sum_{\substack{n\leq X\\(n,q)=1}}
\frac{\mu(n)}{n}\log^k\bigg(\frac{X}{n}\bigg),\qquad\mathbf{\tilde m}_q^{(k)}(X)
=\sum_{\substack{n\leq X\\(n,q)=1}}
\frac{\mu(n)}{\sigma(n)}
\log^k\bigg(\frac{X}{n}\bigg),
\end{align}
where $\mu$ is the M\"obius function and
$\sigma(n)=\sum_{d|n}d$ is the sum-of-divisors function.
The sums are empty when $0<X<1$ and for $k=0$, the logarithmic weight is identically $1$,
including when $n=X$. We also use the notation $
m_q(X)=\mathbf{\check m}_q^{(0)}(X)$,
$\check m_q(X)=\mathbf{\check m}_q^{(1)}(X)$,
$\check{\check m}_q(X)=\mathbf{\check m}_q^{(2)}(X)$ and we omit the subscript when $q=1$.

We derive estimates for these sums of arbitrary nonnegative integer
exponent $k$ from estimates of orders $0,1,2$.
Indeed, by combining M\"obius inversion with the Euler--Maclaurin
summation formula, we obtain a recurrence for
$\mathbf{\check m}^{(k)}(X)$, with a polynomial main term
in $\log X$ and an explicit remainder.
Convolution identities then transfer these estimates
to sums with coprimality conditions and to
$\mathbf{\tilde m}_q^{(k)}$.

Our main estimates are summarised in the following theorem.
For $s\in\mathbb C$ with $\Re s>0$, write
\begin{equation}\label{functions}
\varphi_s:q\in\mathbb{Z}_{\geq 1}\mapsto
q^s\prod_{p|q}\bigg(1-\frac{1}{p^s}\bigg),
\qquad
\mathpzc p_s:q\in\mathbb{Z}_{\geq 1}\mapsto
\prod_{p|q}\frac{p+1}{p+1-p^{1-s}},
\end{equation}
where $\varphi_s$ is the generalised Euler totient function.

\begin{theorem}\label{main_estimates}
Put $\theta=1-\frac{1}{12\log 10}$.
There are explicit positive constants $A_k,B_k$, $\tilde{A}_k,\tilde{B}_k$,
depending only on $k$, such that, for every
$k\in\mathbb{Z}_{\geq 0}$, $q\in\mathbb{Z}_{\geq 1}$
and $X>0$,
\begin{align*}
\big|\mathbf{\check m}_q^{(k)}(X)-Q_{k,q}(\log X)\big|
&\leq
\frac{\sqrt{q}}{\varphi_{1/2}(q)}
\frac{A_k}{\sqrt{X}}
+\frac{q^\theta}{\varphi_\theta(q)}
\frac{B_k\,\mathds{1}_{X\geq 10^{12}}(X)}{\log X},\\
\big|\mathbf{\tilde m}_q^{(k)}(X)
-\widetilde Q_{k,q}(\log X)\big|
&\leq
\mathpzc p_{1/2}(q)\frac{\tilde{A}_k}{\sqrt{X}}
+\mathpzc p_\theta(q)
\frac{\tilde{B}_k\,\mathds{1}_{X\geq 10^{12}}(X)}{\log X}.
\end{align*}
Here $Q_{0,q}=\widetilde Q_{0,q}=0$ and for $k\geq 1$, the polynomials $Q_{k,q}$ and
$\widetilde Q_{k,q}$ have degree $k-1$ and respective
leading coefficients $
\frac{kq}{\varphi(q)}$ and $k\zeta(2)\mathpzc p_1(q)$.
Their coefficients are given by \eqref{Qkq} and
\eqref{tilde_Pol}, and the constants $A_k,B_k$ and $\tilde{A}_k,\tilde{B}_k$ are
specified in \eqref{AB}, \eqref{Ak_rec}, \eqref{Bk_rec} and \eqref{AB_tilde}, respectively.
\end{theorem}

Theorem \ref{main_estimates} follows from
Theorems \ref{coprime_polynomial} and
\ref{tilde_general_bound}, using the constants constructed
in Theorem \ref{general_polynomial}.
In particular, for each fixed $k$ and $q$, both
remainders tend to zero as $X\to\infty$.

Further, for $k=0,1,2$, these two functions have a deep connection with the proof of the ternary Goldbach conjecture by H. Helfgott \cite[\S 5--\S 6]{Helfgott19}. They are also crucial at $k=1,2$ in the author's work
\cite{SZA*22} on the secondary term in the asymptotic
expansion of the Selberg sieve sum
\begin{equation*}
\sum_{\substack{d,e\leq U\\(de,q)=1}}
\frac{\mu(d)\mu(e)}{[d,e]}
\log\bigg(\frac{U}{d}\bigg)
\log\bigg(\frac{U}{e}\bigg),
\qquad U>0,
\end{equation*}
where $[d,e]$ denotes the least common multiple of $d$
and $e$.

There is also a direct connection with the generalised von Mangoldt function $\Lambda_k$.
For $q\in\mathbb Z_{\geq1}$, $k\in\mathbb Z_{\geq1}$ and $X>0$, define
\begin{equation*}
\Lambda_{q,k,X}:n\in\mathbb{Z}_{\geq 1}\mapsto
\sum_{\substack{d|n,\ d\leq X,\ (d,q)=1}}
\mu(d)\log^k\bigg(\frac{X}{d}\bigg).
\end{equation*}
Taking $q=1$ and $X=n$, we recover the generalised
von Mangoldt function. Indeed, $\Lambda_{1,k,n}(n)=\Lambda_k(n)=\sum_{d|n}\mu(d)\log^k\big(\frac{n}{d}\big)$,
which vanishes whenever $n$ has more than $k$ distinct prime factors.
Moreover, for fixed $X>0$, we have
\begin{equation*}
\lim_{N\to\infty}\frac{1}{N}
\sum_{n\leq N}\Lambda_{q,k,X}(n)
=\mathbf{\check m}_q^{(k)}(X).
\end{equation*}

In \cite{RamareZuniga-Alterman*26}, O. Ramar\'e and
the author study related M\"obius sums with a complex
parameter $s$, $\Re s>0$.
In particular, \cite[Thm.~1.2]{RamareZuniga-Alterman*26}
establishes that
\begin{equation*}
\sum_{\substack{n\leq X\\(n,q)=1}}
\frac{\mu(n)}{n^s}\log^k\bigg(\frac{X}{n}\bigg)
\geq 0
\end{equation*}
for $s\in\mathbb{R}_{\geq 1}$, $k,q\in\mathbb{Z}_{\geq 1}$ and $X\geq 1$.
Taking $s=1$ shows that
$\mathbf{\check m}_q^{(k)}(X)\geq 0$ for $k\geq 1$,
also for $0<X<1$ since the sum is empty.
Thus logarithmic smoothing of positive integer order always 
produces nonnegative sums, although the unsmoothed sum
\begin{equation*}
\mathbf{\check m}_q^{(0)}(X)
=\sum_{\substack{n\leq X\\(n,q)=1}}
\frac{\mu(n)}{n}
\end{equation*}
may change sign.
The convolution identity in Lemma \ref{tilde_convolution}
allows us to deduce the corresponding property for the
second family.

\begin{theorem}\label{positive}
Let $k,q\in\mathbb{Z}_{\geq 1}$.
Then, for every $X>0$, we have 
$\mathbf{\tilde m}_q^{(k)}(X)\geq 0$.
\end{theorem}

In \S\ref{Euler--Maclaurin}, we establish the
Euler--Maclaurin estimate underlying the recurrence
in \S\ref{recurrencerelantionship}.
The initial bounds in \S\ref{bounds} then give
Theorem \ref{general_polynomial}.
Subsequently, \S\ref{coprimality} and \S\ref{tilde} extend
these estimates to the two families in
\eqref{definitions}, with Theorem \ref{positive}
proved in the latter section.

The various algebraic calculations appearing in this work were performed using SageMath 10.9 \cite{SageMath*26}.

%%%%%%%%%%%%%%%%%%%%%%%%%%%%%%%%%%%%%%%%%%%%%%%%%%%%%%%%%%%%%
\section{Notation and basic definitions}
%%%%%%%%%%%%%%%%%%%%%%%%%%%%%%%%%%%%%%%%%%%%%%%%%%%%%%%%%%%%%

We define $\log^+:X>0\mapsto\max\{\log X,0\}$ and write $\{t\}=t-\lfloor t\rfloor$ for the fractional
part of $t$.
Throughout the paper, for $T>1$, an expression of the form
$\mathds{1}_{X\geq T}(X)E(X)$ is defined to be zero
when $X<T$, even if $E(X)$ is undefined there.

When discussing series of holomorphic functions, normal convergence means that, on every compact subset of the domain, the sum of the suprema of the absolute values of the terms is finite. Throughout this work, we use normal convergence to justify interchanging sums and differentiating series term by term.

\subsection{Bernoulli polynomials and the Euler--Maclaurin summation formula}

The Bernoulli polynomials are defined by $B_0(x)=1$,
$B_k'(x)=kB_{k-1}(x)$ and $\int_0^1B_k(x)\,dx=0$
for $k\geq 1$. In particular,
$B_1(x)=x-\frac{1}{2}$, 
$B_2(x)=x^2-x+\frac{1}{6}$, 
$B_3(x)=x^3-\frac{3}{2}x^2+\frac{1}{2}x$.
Thus
\begin{equation}\label{BernoulliBounds}
\max_{x\in[0,1]}|B_1(x)|=\frac{1}{2},\qquad
\max_{x\in[0,1]}|B_2(x)|=\frac{1}{6},\qquad
\max_{x\in[0,1]}|B_3(x)|=\frac{1}{12\sqrt{3}}.
\end{equation}

For real numbers $0<x<y$ and $f\in C^2([x,y])$,
the Euler--Maclaurin summation formula of order $2$ gives
\begin{align}\label{EulerMaclaurin}
\sum_{x<n\leq y}f(n)
={}&\int_x^y f(t)\,dt
-B_1(\{y\})f(y)+B_1(\{x\})f(x)\nonumber\\
&+\frac{1}{2}B_2(\{y\})f'(y)
-\frac{1}{2}B_2(\{x\})f'(x)
-\frac{1}{2}\int_x^y B_2(\{t\})f''(t)\,dt.
\end{align}

\subsection{Stieltjes constants}

The Stieltjes constants $\gamma_j$ are defined by the
Laurent expansion of the Riemann $\zeta$ function at $1$ as follows
\begin{equation*}
\zeta(s)=\frac{1}{s-1}
+\sum_{j=0}^{\infty}
\frac{(-1)^j\gamma_j}{j!}(s-1)^j.
\end{equation*}
Here $\gamma_0=\gamma$ is the Euler--Mascheroni constant.
Equivalently, we have
\begin{equation}\label{gamma:limit}
\gamma_j
=\lim_{X\to\infty}
\bigg[
\sum_{n\leq X}\frac{\log^j n}{n}
-\frac{\log^{j+1}X}{j+1}
\bigg],
\qquad j\geq 0;
\end{equation}
see \cite{BriggsChowla*55}. The first Stieltjes constants are
\begin{table}[ht]
\centering
\renewcommand{\arraystretch}{1.4}
\begin{tabular}{|r@{${}={}$}r|r@{${}={}$}r|
                r@{${}={}$}r|}
\hline
$\gamma_0$ & $0.5772156649\ldots$ &
$\gamma_1$ & $-0.0728158454\ldots$ &
$\gamma_2$ & $-0.0096903631\ldots$\\
\hline
$\gamma_3$ & $0.0020538344\ldots$ &
$\gamma_4$ & $0.0023253700\ldots$ &
$\gamma_5$ & $0.0007933238\ldots$\\
\hline
$\gamma_6$ & $-0.0002387693\ldots$ &
$\gamma_7$ & $-0.0005272895\ldots$ &
$\gamma_8$ & $-0.0003521233\ldots$\\
\hline
\end{tabular}
\end{table}

%%%%%%%%%%%%%%%%%%%%%%%%%%%%%%%%%%%%%%%%%%%%%%%%%%%%%%%%%%%%%%
\section{An application of the Euler--Maclaurin summation formula}\label{Euler--Maclaurin} 
%%%%%%%%%%%%%%%%%%%%%%%%%%%%%%%%%%%%%%%%%%%%%%%%%%%%%%%%%%%%%%

\begin{lemma}\label{EM}Let $X>0$ and $j\in\mathbb{Z}_{\geq 0}$. Then
\begin{equation}\label{R_j}
\sum_{m\leq X}\frac{\log^jm}{m}=\frac{\log^{j+1}X}{j+1}+\gamma_j+R_j(X)
\end{equation}
where $\gamma_j$ is the $j$-th Stieltjes constant and 
\begin{equation*}
|R_j(X)|\leq \frac{|\log X|^j}{2X}+\frac{P_j(|\log X|)}{X^2}
\end{equation*}
with $P_j$ defined by
\begin{equation}\label{Pol}
P_0:t\geq 0\mapsto\frac{1}{6},\qquad P_j:t\geq 0\mapsto\frac{t^j}{6}+\frac{jt^{j-1}}{4}+\frac{j!}{2}\sum_{i=2}^{j}\frac{t^{j-i}}{(j-i)!2^{i}},\qquad j\geq 1.
\end{equation}
\end{lemma}

\begin{proofbold} Suppose that $X\geq 1$. By applying the Euler--Maclaurin summation formula \eqref{EulerMaclaurin} (of order $2$) to the function $\ell_j:t>0\mapsto \frac{\log^jt}{t}$, we obtain
\begin{align}\label{eq:log_mc}
\sum_{m\leq X}\frac{\log^jm}{m}&=\mathds{1}_{\{j=0\}}(j)+\sum_{1<m\leq X}\frac{\log^jm}{m}=\frac{\log^{j+1}X}{j+1}+I_j+R_j(X),
\end{align}
where 
\begin{equation*}
I_j=\frac{7\ \mathds{1}_{\{j=0\}}(j)}{12}-\frac{\mathds{1}_{\{j=1\}}(j)}{12}-\frac{1}{2}\int_1^\infty B_2(\{t\})\ell_j''(t)dt\\
\end{equation*}
and
\begin{equation}\label{R_j_alt}
R_j(X)=-B_1(\{X\})\ell_j(X)+\frac12 B_2(\{X\})\ell_j'(X)+\frac{1}{2}\int_X^\infty B_2(\{t\})\ell_j''(t)\ dt.
\end{equation}
Observe that
\begin{align*}
\ell_0'(t)&=\frac{-1}{t^2},\quad\ell_1'(t)=\frac{1-\log t}{t^2},\quad\ell_j'(t)=\frac{j\log^{j-1}t-\log^jt}{t^2},\qquad j\geq 2,\\
\ell_0''(t)&=\frac{2}{t^3},\quad\ell_1''(t)=\frac{2\log t-3}{t^3},\quad\ell_j''(t)=\frac{\log^{j-2}t\ (2\log^2t-3j\log t+j(j-1))}{t^3},\qquad j\geq 2.
\end{align*}
Thus, since estimates \eqref{BernoulliBounds} show that $B_1$ and $B_2$ are bounded functions on $[0,1]$, we have $R_j(X)=O_j(\frac{\log^jX}{X})$ and the integral defining $I_j$ converges. In fact, by \eqref{gamma:limit}, we have $I_j=\gamma_j$, where $\gamma_j$ is the $j$-th Stieltjes constant. 

Now, by using \eqref{BernoulliBounds} again in \eqref{R_j_alt}, we obtain
\begin{equation}\label{eq:Rj-basic}
|R_j(X)|
\leq\frac{\log^jX}{2X}
+\frac{|\ell_j'(X)|}{12}
+\frac1{12}\int_X^\infty|\ell_j''(t)|\ dt.
\end{equation}
For $j=0$, the above reads as
\begin{equation*}
|R_0(X)|
\leq\frac{1}{2X}+\frac{1}{12X^2}
+\frac{1}{12}\int_X^\infty\frac{2}{t^3}\ dt
=\frac{1}{2X}+\frac{1}{6X^2},
\end{equation*}
which gives the claimed bound with $P_0$. Suppose now that $j\geq 1$; then we conclude from \eqref{eq:Rj-basic}
\begin{align}\label{eq:Rj-integrals}
|R_j(X)|
\leq&
\frac{\log^jX}{2X}
+\frac{\log^jX+j\log^{j-1}X}{12X^2}\\
&+\frac1{12}\bigg(
2\int_X^\infty\frac{\log^jt}{t^3}\ dt+3j\int_X^\infty\frac{\log^{j-1}t}{t^3}\ dt+j(j-1)\int_X^\infty\frac{\log^{j-2}t}{t^3}\ dt
\bigg).\nonumber
\end{align}
Note that the above last expression does not vanish only when $j\geq 2$. Moreover, integration by parts yields
\begin{equation*}
\int_X^\infty\frac{\log^0t}{t^3}\ dt=\frac{1}{2X^2},
\qquad
\int_X^\infty\frac{\log^rt}{t^3}\ dt=\frac{\log^rX}{2X^2}+\frac{r}{2}\int_X^\infty\frac{\log^{r-1}t}{t^3}\ dt,
\qquad r\geq 1.
\end{equation*}
Hence, by induction, for any $r\in\mathbb{Z}_{\geq 0}$,
\begin{equation}\label{eq:Jr}
\int_X^\infty\frac{\log^rt}{t^3}\ dt
=\frac{1}{2X^2}\sum_{i=0}^r
\frac{r!\ \log^{r-i}X}{(r-i)!\ 2^{i}}.
\end{equation}

Substituting \eqref{eq:Jr} into \eqref{eq:Rj-integrals},
the coefficient of $\frac{\log^jX}{X^2}$ is $
\frac{1}{12}+\frac{1}{12}=\frac{1}{6}$,
and that of $\frac{\log^{j-1}X}{X^2}$ is $\frac{j}{12}+\frac{j}{24}+\frac{j}{8}=\frac{j}{4}$.
For $2\leq i\leq j$, the coefficient of $\frac{\log^{j-i}X}{X^2}$ is
\begin{equation*}
\frac{1}{12}\bigg(
\frac{j!}{(j-i)!\ 2^{i}}+\frac{3j!}{(j-i)!\ 2^{i}}+\frac{j!}{(j-i)!\ 2^{i-1}}\bigg)
=\frac{j!}{2(j-i)!\ 2^{i}}.
\end{equation*}
Therefore
\begin{equation*}
|R_j(X)|
\leq\frac{\log^jX}{2X}+\frac{P_j(\log X)}{X^2},
\end{equation*}
where $P_j$ is defined by \eqref{Pol}.

Finally, suppose that $0<X<1$. The Euler--Maclaurin remainder identity \eqref{eq:log_mc}
remains valid for every $X>0$, with the leftmost sum
being empty. For every integer $r\geq 0$, the
substitution $t=Xe^v$ gives
\begin{align*}
\int_X^\infty\frac{|\log t|^r}{t^3}\ dt
=\frac{1}{X^2}\int_0^\infty
|\log X+v|^r e^{-2v}\ dv&\leq\frac{1}{X^2}\int_0^\infty
\big(|\log X|+v\big)^r e^{-2v}\ dv\\
&=\frac{1}{2X^2}\sum_{i=0}^{r}
\frac{r!}{(r-i)!2^i}|\log X|^{r-i}.
\end{align*}
Hence, applying the above estimate to the derivatives of
$\ell_j(t)=\log^jt/t$ in \eqref{eq:Rj-basic} via \eqref{eq:Rj-integrals}, we obtain the same
polynomial $P_j$, first for $j=0$ and then evaluated at
$|\log X|$ for $j\geq 1$. 
\end{proofbold}

%%%%%%%%%%%%%%%%%%%%%%%%%%%%%%%%%%%%%%%%%%%%%%%%%%%%%%%%%%%%%%
\section{A recurrence relationship}\label{recurrencerelantionship}
%%%%%%%%%%%%%%%%%%%%%%%%%%%%%%%%%%%%%%%%%%%%%%%%%%%%%%%%%%%%%%

We first derive estimates for the functions $\mathbf{\check m}^{(k)}(X)$, $k\in\mathbb{Z}_{\geq0}$.
We will then extend these estimates to $\mathbf{\check m}_q^{(k)}(X)$ in \S\ref{coprimality}
and to $\mathbf{\tilde m}_q^{(k)}(X)$ in \S\ref{tilde}.

Consider the family $\mathpzc F$ of functions
$(0,\infty)\to\mathbb C$ that vanish on $(0,1)$.
For example, the function
$\log^+:X>0\mapsto\max\{\log X,0\}$
belongs to this family. Following \cite[\S 2]{BA12}, define two linear maps
$\mathpzc F\to\mathpzc F$ as
\begin{equation}
S_1(f):X>0\mapsto\sum_{n=1}^{\infty}
f\bigg(\frac{X}{n}\bigg),
\qquad
S_\mu(f):X>0\mapsto\sum_{n=1}^{\infty}
\mu(n)f\bigg(\frac{X}{n}\bigg)
\end{equation}
which are well defined finite sums, since only $n\leq X$ contribute.
M\"obius inversion gives the following result.

\begin{lemma}\label{mobinversion}
The maps $S_1$ and $S_\mu$ are inverse to each other:
for every $f\in\mathpzc F$,
\begin{equation}
f=S_\mu(S_1(f))=S_1(S_\mu(f)).
\end{equation}
\end{lemma}

Now, let $k\in\mathbb Z_{\geq 2}$. Apply Lemma
\ref{mobinversion} to
$f_k=\mathrm{Id}\cdot\log^{k-2}\cdot\log^+$,
where $\mathrm{Id}$ is the identity function and
$\log^0=\mathbf 1$. Thus, for every $X>0$, we obtain
\begin{equation*}
X\log^{k-2}X\log^+X
=\sum_{n\leq X}\mu(n)
\sum_{m\leq\frac{X}{n}}\frac{X}{mn}
\log^{k-2}\bigg(\frac{X}{mn}\bigg)
\log^+\bigg(\frac{X}{mn}\bigg).
\end{equation*}
Dividing by $X>0$ and applying the binomial theorem gives
\begin{equation}\label{id2}
\log^{k-2}X\log^+X
=\sum_{n\leq X}\frac{\mu(n)}{n}
\sum_{j=0}^{k-1}(-1)^j\mbinom{k-1}{j}
\log^{k-1-j}\bigg(\frac{X}{n}\bigg)
\sum_{m\leq\frac{X}{n}}\frac{\log^jm}{m}.
\end{equation}

Recall the definition \eqref{R_j} of $R_j$.
For every integer $r\geq 0$, define
\begin{equation}\label{S_r}
S_r:t\geq 1\mapsto
\sum_{j=0}^{r}(-1)^j\mbinom{r}{j}
\log^{r-j}t\ R_j(t).
\end{equation}
Applying Lemma \ref{EM} to the innermost sum in
\eqref{id2}, and using
\begin{equation*}
\sum_{j=0}^{k-1}\mbinom{k-1}{j}\frac{(-1)^j}{j+1}
=\int_0^1\sum_{j=0}^{k-1}\mbinom{k-1}{j}(-t)^j\ dt
=\int_0^1(1-t)^{k-1}\ dt
=\frac{1}{k},
\end{equation*}
we obtain
\begin{align*}
\log^{k-2}X\log^+X
&=\sum_{n\leq X}\frac{\mu(n)}{n}
\sum_{j=0}^{k-1}(-1)^j\mbinom{k-1}{j}
\log^{k-1-j}\bigg(\frac{X}{n}\bigg)
\bigg(
\frac{\log^{j+1}(X/n)}{j+1}
+\gamma_j+R_j\bigg(\frac{X}{n}\bigg)
\bigg)\nonumber\\
&=\frac{\mathbf{\check m}^{(k)}(X)}{k}
+\sum_{j=0}^{k-2}(-1)^j\mbinom{k-1}{j}
\gamma_j\ \mathbf{\check m}^{(k-1-j)}(X)
-\frac{\Xi_k(X)}{k},
\end{align*}
where
\begin{equation}\label{Xi}
\Xi_k(X)
=(-1)^k k\gamma_{k-1}m(X)
-k\sum_{n\leq X}\frac{\mu(n)}{n}
S_{k-1}\bigg(\frac{X}{n}\bigg).
\end{equation}
We therefore obtain the following result.

\begin{theorem}\label{recurrence}
Let $k\in\mathbb Z_{\geq 2}$. For every $X>0$,
\begin{equation*}
\mathbf{\check m}^{(k)}(X)
=k\log^{k-2}X\log^+X
-k\sum_{j=1}^{k-1}\mbinom{k-1}{j}
(-1)^{k-1-j}\gamma_{k-1-j}
\mathbf{\check m}^{(j)}(X)
+\Xi_k(X),
\end{equation*}
where $\Xi_k(X)$ is defined in \eqref{Xi}. 

Moreover,
\begin{align*}
|\Xi_k(X)|
&\leq k(k-1)\int_1^{\max\{1,X\}}
\bigg|m\bigg(\frac{X}{t}\bigg)\bigg|W_{k-2}(t)\ dt,
\end{align*}
where, for $r\geq 0$,
\begin{equation}\label{W_r}
W_r:t\geq 1\mapsto
\frac{2^{r-1}\log^r t}{t^2}
+\frac{1}{t^3}\sum_{j=0}^{r}
\mbinom{r}{j}\log^{r-j}t\ P_j(\log t),
\end{equation}
and the polynomials $P_j$ are defined by \eqref{Pol} 
\end{theorem}

\begin{proofbold}
If $0<X<1$, then $\Xi_k(X)=0$ and the integral is zero. We may therefore suppose that $X\geq1$.

By Lemma \ref{EM}, the triangle inequality and the
binomial theorem, for every $r\geq 0$ and $t\geq 1$,
\begin{equation}\label{S_r:bound}
|S_r(t)|
\leq tW_r(t).
\end{equation}
Moreover, the definition of $R_j$ gives another expression for $S_r$ defined by \eqref{S_r}, namely
\begin{equation*}
S_r(t)
=\sum_{m\leq t}\frac{\log^r(t/m)}{m}
-\frac{\log^{r+1}t}{r+1}
-\sum_{j=0}^{r}(-1)^j\mbinom{r}{j}
\gamma_j\log^{r-j}t.
\end{equation*}
Further, for $r\geq 1$, the function $S_r$ is continuous
and continuously differentiable between consecutive
integers, with
\begin{equation*}
S_r(1)=(-1)^{r+1}\gamma_r,
\qquad
S_r'(t)=\frac{r}{t}S_{r-1}(t),
\qquad t\geq 1,\quad t\notin\mathbb Z.
\end{equation*}
Consequently, since $k\geq 2$, partial summation applies to
$S_{k-1}$ and yields
\begin{align*}
\Xi_k(X)
&=(-1)^k k\gamma_{k-1}m(X)
-km(X)S_{k-1}(1)+k\int_1^X m(u)
\frac{\mathrm{d}}{\mathrm{d}u}
\bigg[S_{k-1}\bigg(\frac{X}{u}\bigg)\bigg]\ du\\
&\qquad\quad=-k(k-1)\int_1^X m(u)
S_{k-2}\bigg(\frac{X}{u}\bigg)\frac{du}{u}\ =-k(k-1)\int_1^X
m\bigg(\frac{X}{t}\bigg)S_{k-2}(t)\frac{dt}{t}.
\end{align*}
Taking absolute values and applying \eqref{S_r:bound}
with $r=k-2$ proves the result.
\end{proofbold}

%%%%%%%%%%%%%%%%%%%%%%%%%%%%%%%%%%%%%%%%%%%%%%%%%%%%%%%%%%%%%%
\section{Initial estimates}\label{bounds} 
%%%%%%%%%%%%%%%%%%%%%%%%%%%%%%%%%%%%%%%%%%%%%%%%%%%%%%%%%%%%%%

In order to apply the recurrence result given by Theorem \ref{recurrence}, we need the following initial estimates. 
\begin{lemma}\label{initial_estimates}
For every $X>0$, we have
\begin{align}
|m(X)|\qquad\quad
&\leq\frac{\sqrt{2}}{\sqrt{X}}
+\frac{1.862\ \mathds{1}_{{X\geq 10^{12}}}(X)}
{10^4\log X},\label{mlog^2}\\
|\check m(X)-1|\qquad
&\leq\frac{1}{\sqrt{X}}
+\frac{8.55\ 
\mathds{1}_{{X\geq 10^{12}}}(X)}
{10^6\log X},\label{checklog^2}\\
|\check{\check m}(X)-2\log X+2\gamma|
&\leq\frac{4e^{\frac{\gamma}{2}-1}}{\sqrt{X}}
+\frac{2.989\ \mathds{1}_{{X\geq 10^{12}}}(X)}
{10^4\log X}.\label{checkcheckmlog^2}
\end{align}
\end{lemma}
\begin{proofbold}
We have
\begin{equation*}
|m(X)|\leq
\begin{cases}
\displaystyle\ \frac{3}{\pi 10^{10}}+\frac{11.39}{\sqrt{X}},
&\text{if }X\geq 1,
\quad\text{\cite[Cor. 1.3]{ChirreHelfgott*25a}},\\[8pt]\ 
\displaystyle\frac{362.84}{\log^2X},
&\text{if }X\geq 2,
\quad\text{\cite[Thm. A]{Daval25}},\\[8pt]\ 
\displaystyle\frac{\sqrt{2}}{\sqrt{X}},
&\text{if }0<X\leq10^{14},
\quad\text{\cite[Lemma~5.10]{Helfgott19}}.
\end{cases}
\end{equation*}
We readily verify that the function
\begin{equation*}
f:t>0\mapsto\frac{3t}{\pi10^{10}}+\frac{(11.39-\sqrt{2})t}{e^{t/2}}
\end{equation*}
is convex on $[14\log 10,1.949\cdot 10^6]$. Moreover, 
\begin{equation*}
f(14\log 10)<0.00003217\quad\text{and}\quad f(1.949\cdot 10^6)<0.00018612.
\end{equation*}
Thus, if we suppose that $X\in (10^{14},e^{1.949\cdot 10^6}]$, then $f(\log X)<\frac{1.862}{10^4}$,
which gives \eqref{mlog^2} in that interval.
For $X\geq e^{1.949\cdot 10^6}$, we instead use
\begin{equation*}
|m(X)|
\leq\frac{362.84}{\log^2X}
\leq\frac{362.84}{1.949\cdot 10^6\log X}
<\frac{1.862}{10^4\log X}.
\end{equation*}
This proves \eqref{mlog^2} for every $X>0$.

Estimate \eqref{checklog^2} follows from
\cite[Lemma~5.9, Eq.(5.54)]{Helfgott19} for $X\leq 10^{12}$ and
\cite[Thm.~B]{Daval25} for $X\geq 10^{12}$.

For the last estimate, the exact remainder identity
in the proof of Theorem \ref{recurrence}, with $k=2$,
gives, for $X\geq 1$,
\begin{equation}\label{R_0}
\check{\check m}(X)-2\log X+2\gamma
=-2\gamma(\check m(X)-1)-2\int_1^X m\bigg(\frac{X}{t}\bigg)
R_0(t)\frac{dt}{t}.
\end{equation}
Here
\begin{equation*}
\frac{|R_0(t)|}{t}
\leq W_0(t)=\frac{1}{2t^2}+\frac{1}{6t^3},
\end{equation*}
as defined in \eqref{W_r}.
Also, observe that
\begin{equation*}
\int_1^\infty\sqrt{t}\ W_0(t)\ dt=\frac{10}{9},
\qquad
\int_1^\infty W_0(t)\ dt=\frac{7}{12}
\quad\text{and}\quad
\int_1^\infty W_0(t)\log t\ dt=\frac{13}{24}.
\end{equation*}
Applying \eqref{mlog^2} into \eqref{R_0}, and using the inequality
\begin{equation}\label{inek}
\frac{1}{\log(X/t)}
\leq\frac{1}{\log X}
\bigg(1+\frac{\log t}{12\log 10}\bigg),
\qquad1\leq t\leq \frac{X}{10^{12}},
\end{equation}
we obtain
\begin{align*}
&|\check{\check m}(X)-2\log X+2\gamma|-2\gamma|\check m(X)-1|
\leq 2\int_1^X\bigg|m\bigg(\frac{X}{t}\bigg)\bigg|\frac{|R_0(t)|}{t}dt\\
&\qquad\leq \frac{2\sqrt{2}}{\sqrt{X}}\int_1^X\sqrt{t}W_0(t)dt
+2\cdot\frac{1.862}{10^4}\ \mathds{1}_{{X\geq 10^{12}}}(X)\int_1^{X/10^{12}}\frac{W_0(t)}{\log(X/t)}dt\\
&\qquad\qquad\leq \frac{2\sqrt{2}}{\sqrt{X}}\int_1^X\sqrt{t}W_0(t)dt
+2\cdot\frac{1.862}{10^4}\frac{\mathds{1}_{{X\geq 10^{12}}}(X)}{\log X}\int_1^{X/10^{12}}\bigg(1+\frac{\log t}{12\log 10}\bigg)W_0(t)dt.\end{align*}
Therefore, adding $2\gamma|\check m(X)-1|$ to the above estimate and
applying \eqref{checklog^2} give
\begin{equation*}
|\check{\check m}(X)-2\log X+2\gamma|\leq\frac{C}{\sqrt{X}}+\frac{D\ \mathds{1}_{{X\geq 10^{12}}}(X)}{\log X},
\end{equation*}
where
\begin{equation*}
C=2\gamma+\frac{20\sqrt{2}}{9},\qquad D=2\gamma\frac{8.55}{10^6}+\frac{1.862}{10^4}\bigg(\frac{7}{6}+\frac{13}{144\log 10}\bigg).
\end{equation*}
Since $C>4e^{\frac{\gamma}{2}-1}$ and the function $t\geq e^2\mapsto\log t/\sqrt{t}$ is decreasing, we have, for $X\geq 10^{12}$,
\begin{equation*}
\frac{C}{\sqrt{X}}+\frac{D}{\log X}
\leq\frac{4e^{\frac{\gamma}{2}-1}}{\sqrt{X}}
+\bigg(D+(C-4e^{\frac{\gamma}{2}-1})\frac{12\log 10}{\sqrt{10^{12}}}\bigg)\frac{1}{\log X}.
\end{equation*}
Finally,
\begin{equation*}
D+\bigg(C-4e^{\frac{\gamma}{2}-1}\bigg)
\frac{12\log 10}{10^6}
<\frac{2.98876}{10^4}<\frac{2.989}{10^4}.
\end{equation*}
For the remaining case $0<X<10^{12}$, the required bound is already supplied
by \cite[Lemma~5.9, Eq.(5.56)]{Helfgott19}. This proves
\eqref{checkcheckmlog^2}.
\end{proofbold}

%%%%%%%%%%%%%%%%%%%%%%%%%%%%%%%%%%%%%%%%%%%%%%%%%%%%%%%%%%%%%%
\section{Effective estimates}
%%%%%%%%%%%%%%%%%%%%%%%%%%%%%%%%%%%%%%%%%%%%%%%%%%%%%%%%%%%%%%

\subsection{A recursion on polynomials}

Set $Q_0(t)=0$ and, recursively for every integer $k\geq 1$, define
\begin{equation}\label{Pol_2}
Q_k(t)=kt^{k-1}-k\sum_{j=1}^{k-1}\mbinom{k-1}{j}(-1)^{k-1-j}\gamma_{k-1-j}Q_j(t),
\end{equation}
where the sum is empty when $k=1$.
 Moreover, differentiating \eqref{Pol_2} and arguing by induction gives
$Q_k'(t)=kQ_{k-1}(t)$ for every $k\geq 1$. The first eight polynomials are

\begin{align*}
&Q_0(t)=0,\qquad Q_1(t)=\phantom{.}1,\qquad Q_2(t)=\phantom{.}2t-2\gamma,\qquad Q_3(t)=\phantom{.}3t^2-6\gamma t+6(\gamma^2+\gamma_1),\\
 &Q_4(t)=4t^3-12\gamma t^2
+24(\gamma^2+\gamma_1)t-12(2\gamma^3+4\gamma\gamma_1+\gamma_2),\\
 &Q_5(t) =5t^4-20\gamma t^3
+60(\gamma^2+\gamma_1)t^2-60(2\gamma^3+4\gamma\gamma_1+\gamma_2)t\\
&\qquad\qquad+20(6\gamma^4+18\gamma^2\gamma_1
+6\gamma_1^2+6\gamma\gamma_2+\gamma_3),\\
&Q_6(t) =6t^5-30\gamma t^4+120(\gamma^2+\gamma_1)t^3-180(2\gamma^3+4\gamma\gamma_1+\gamma_2)t^2\\
&\qquad\qquad+120(6\gamma^4+18\gamma^2\gamma_1
+6\gamma_1^2+6\gamma\gamma_2+\gamma_3)t\\
&\qquad\qquad\qquad-30\big(24\gamma^5+96\gamma^3\gamma_1
+72\gamma\gamma_1^2+36\gamma^2\gamma_2+24\gamma_1\gamma_2+8\gamma\gamma_3+\gamma_4\big),\\
&Q_7(t) =7t^6-42\gamma t^5
+210(\gamma^2+\gamma_1)t^4-420\big(2\gamma^3+4\gamma\gamma_1+\gamma_2\big)t^3\\
&\qquad\qquad+420\big(6\gamma^4+18\gamma^2\gamma_1
+6\gamma_1^2+6\gamma\gamma_2+\gamma_3\big)t^2\\
&\qquad\qquad\qquad-210\big(24\gamma^5+96\gamma^3\gamma_1+72\gamma\gamma_1^2+36\gamma^2\gamma_2+24\gamma_1\gamma_2+8\gamma\gamma_3+\gamma_4\big)t\\
&\qquad\qquad\qquad\qquad+42\big(120\gamma^6+600\gamma^4\gamma_1
+720\gamma^2\gamma_1^2+120\gamma_1^3+240\gamma^3\gamma_2\big)\\
&\qquad\qquad\qquad\qquad\qquad+42\big(360\gamma\gamma_1\gamma_2+30\gamma_2^2+60\gamma^2\gamma_3+40\gamma_1\gamma_3+10\gamma\gamma_4+\gamma_5\big).
\end{align*}

\subsection{General estimates}

With the help of the recurrence relationship we have described in Theorem \ref{recurrence}, we will construct estimates of the shape \eqref{general_polynomial_bound} for $k\geq 1$. In fact, thanks to Lemma \ref{initial_estimates}, we have bounds for $\mathbf{\check m}^{(0)}(X)$, $\mathbf{\check m}^{(1)}(X)$ and $\mathbf{\check m}^{(2)}(X)$, $X>0$, with respective coefficients
\begin{equation}\label{AB}
(A_0,B_0)=\bigg(\sqrt{2},\frac{1.862}{10^4}\bigg),\qquad
(A_1,B_1)=\bigg(1,\frac{8.55}{10^6}\bigg),\qquad
(A_2,B_2)=\bigg(4e^{\frac{\gamma}{2}-1},\frac{2.989}{10^4}\bigg).
\end{equation}
\begin{theorem}\label{general_polynomial}
Define
\begin{align*}
U_r&=r!\bigg(4^r+\frac{14}{45}\bigg(\frac{4}{3}\bigg)^r-\frac{1}{5\cdot2^r}\bigg),\\
V_r&=r!\bigg(2^{r-1}+\frac{1}{3}-\frac{1}{2^{r+2}}+\bigg(2^{r-1}(r+1)+\frac{2r-1}{12}+\frac{1}{2^{r+3}}\bigg)\frac{1}{12\log 10}\bigg),
\qquad r\geq 1.
\end{align*}
For $k\geq3$, let
\begin{align}
A_k
&=k\sum_{j=1}^{k-1}\mbinom{k-1}{j}
|\gamma_{k-1-j}|A_j
+k(k-1)\sqrt{2}\ U_{k-2},\label{Ak_rec}\\
B_k
&=k\sum_{j=1}^{k-1}\mbinom{k-1}{j}
|\gamma_{k-1-j}|B_j
+\frac{1.862}{10^4}k(k-1)V_{k-2}.
\label{Bk_rec}
\end{align}
Then, for every integer $k\geq 0$ and every $X>0$,
\begin{equation}\label{general_polynomial_bound}
\big|\mathbf{\check m}^{(k)}(X)-Q_k(\log X)\big|
\leq\frac{A_k}{\sqrt{X}}
+\frac{B_k\ \mathds{1}_{X\geq 10^{12}}(X)}{\log X},
\end{equation}
where $Q_0=0$ and, for $k\geq 1$, $Q_k$ is the
polynomial of degree $k-1$ and leading coefficient $k$
defined by \eqref{Pol_2}.
\end{theorem}

\begin{proofbold}
The cases $k=0,1,2$ follow from Lemma
\ref{initial_estimates} and \eqref{AB}, since
$Q_0(t)=0$, $Q_1(t)=1$ and $Q_2(t)=2t-2\gamma$.

For $k\geq 3$ we use induction: suppose that the result holds for all positive integers $1\leq j\leq k-1$ and all $X>0$.

Note now the three following binomial identities
\begin{equation*}
\sum_{j=0}^{r}\mbinom{r}{j}=2^r,\quad j\mbinom{r}{j}=r\mbinom{r-1}{j-1},\quad
\sum_{j=i}^{r}\mbinom{r}{j}\frac{j!}{(j-i)!}
=\frac{r!}{(r-i)!}
\sum_{j=i}^{r}\mbinom{r-i}{j-i}
=\frac{r!}{(r-i)!}2^{r-i}.
\end{equation*}
Consequently, the definition \eqref{Pol} of the polynomials $P_j$ yields
\begin{align*}
\sum_{j=0}^{r}\mbinom{r}{j}t^{r-j}P_j(t)
&=\frac{2^r}{6}t^r+\frac{r2^{r-1}}{4}t^{r-1}+\frac{1}{2}\sum_{i=2}^{r}\frac{t^{r-i}}{2^i}
\sum_{j=i}^{r}\mbinom{r}{j}\frac{j!}{(j-i)!}\\
&=\frac{2^r}{6}t^r
+\sum_{i=1}^{r}\frac{r!}{(r-i)!}
2^{r-2i-1}t^{r-i},\qquad r\geq 1.
\end{align*}
Recall the definition \eqref{W_r} of $W_r$. The preceding calculation gives, for $r\geq1$,
\begin{equation*}
W_r(t)=\frac{2^{r-1}\log^r t}{t^2}+\frac{2^r\log^r t}{6t^3}
+\sum_{i=1}^r\frac{r!}{(r-i)!}\,2^{r-2i-1}\frac{\log^{r-i}t}{t^3}.
\end{equation*}
We integrate each term using $
\int_1^\infty\frac{\log^\ell t}{t^{1+\alpha}}\,dt
=\frac{\ell!}{\alpha^{\ell+1}}$, $\ell\in\mathbb Z_{\geq 0}$, $\alpha>0$. This gives
\begin{align*}
\int_1^\infty\sqrt t\,W_r(t)\,dt
&=r!\bigg(4^r+\frac19\bigg(\frac43\bigg)^r+\frac13\bigg(\frac43\bigg)^r\sum_{i=1}^r\bigg(\frac38\bigg)^i\bigg)=r!\bigg(4^r+\frac{14}{45}\bigg(\frac43\bigg)^r-\frac1{5\cdot2^r}\bigg)=U_r,\\
\int_1^\infty W_r(t)\,dt
&=r!\bigg(2^{r-1}+\frac1{12}+\frac14\sum_{i=1}^r\frac1{2^i}\bigg)=r!\bigg(2^{r-1}+\frac13-\frac1{2^{r+2}}\bigg),\\
\int_1^\infty W_r(t)\log t\,dt
&=r!\bigg(2^{r-1}(r+1)+\frac{r+1}{24}+\frac18\sum_{i=1}^r\frac{r-i+1}{2^i}\bigg)\\
&\qquad=r!\bigg(2^{r-1}(r+1)+\frac{2r-1}{12}+\frac1{2^{r+3}}\bigg).
\end{align*}
Consequently,
\begin{equation*}
\int_1^\infty W_r(t)\bigg(1+\frac{\log t}{12\log10}\bigg)\,dt=V_r.
\end{equation*}
Suppose first that $X\geq 1$. By estimate \eqref{mlog^2},
\begin{align*}
\int_1^X
\bigg|m\bigg(\frac{X}{t}\bigg)\bigg|W_r(t)\ dt
\leq\phantom{.}&
\frac{\sqrt{2}}{\sqrt{X}}
\int_1^X\sqrt{t}\ W_r(t)\ dt+\frac{1.862\cdot\mathds{1}_{X\geq 10^{12}}}{10^4}
\int_1^{X/10^{12}}\frac{W_r(t)}{\log(X/t)}\ dt.
\end{align*}
By recalling inequality \eqref{inek}, we then have
\begin{equation*}
\int_1^X
\bigg|m\bigg(\frac{X}{t}\bigg)\bigg|W_r(t)\ dt
\leq\frac{\sqrt{2}\ U_r}{\sqrt{X}}
+\frac{1.862\ V_r\ \mathds{1}_{{X\geq 10^{12}}}(X)}
{10^4\log X}.
\end{equation*}
Thus, Theorem \ref{recurrence} with $r=k-2$ gives
\begin{equation*}
|\Xi_k(X)|
\leq\frac{k(k-1)\sqrt{2}\ U_{k-2}}{\sqrt{X}}
+\frac{1.862\ k(k-1)V_{k-2}\ 
\mathds{1}_{{X\geq 10^{12}}}(X)}{10^4\log X}.
\end{equation*}

On the other hand, subtracting \eqref{Pol_2} evaluated at $\log X$
from the recurrence in Theorem \ref{recurrence} gives
\begin{equation*}
\mathbf{\check m}^{(k)}(X)-Q_k(\log X)
=-k\sum_{j=1}^{k-1}\mbinom{k-1}{j}
(-1)^{k-1-j}\gamma_{k-1-j}
\big(\mathbf{\check m}^{(j)}(X)-Q_j(\log X)\big)
+\Xi_k(X).
\end{equation*}
Applying the induction hypothesis and the preceding
bound for $\Xi_k(X)$ proves
\eqref{general_polynomial_bound} for $X\geq 1$,
with the constants \eqref{Ak_rec} and \eqref{Bk_rec}.

Finally, for $0<X<1$, we have $\mathbf{\check m}^{(k)}(X)=0$.
Hence, by \eqref{Pol_2} and the induction hypothesis,
\begin{align*}
\big|\mathbf{\check m}^{(k)}(X)-Q_k(\log X)\big|
&=|Q_k(\log X)|\\
&\leq k|\log X|^{k-1}
+k\sum_{j=1}^{k-1}\mbinom{k-1}{j}
|\gamma_{k-1-j}||Q_j(\log X)|\\
&= k|\log X|^{k-1}
+k\sum_{j=1}^{k-1}\mbinom{k-1}{j}
|\gamma_{k-1-j}||\mathbf{\check m}^{(j)}(X)-Q_j(\log X)|\\
&\leq \frac{1}{\sqrt{X}}\bigg(k\bigg(\frac{2(k-1)}{e}\bigg)^{k-1}
+k\sum_{j=1}^{k-1}\mbinom{k-1}{j}
|\gamma_{k-1-j}|A_j\bigg),
\end{align*}
where we have used that the function
$t\mapsto\log^{k-1}t/\sqrt{t}$ is increasing on
$[1,e^{2(k-1)}]$ and decreasing on $[e^{2(k-1)},\infty)$, so that
\begin{equation*}
\sqrt{X}\,|\log X|^{k-1}
\leq\bigg(\frac{2(k-1)}{e}\bigg)^{k-1},
\qquad 0<X<1.
\end{equation*}
Since $U_{k-2}\geq 4^{k-2}(k-2)!$, $k\geq3$, and $k!\geq (k/e)^k$, $k\geq 1$, we have
\begin{equation*}
k(k-1)\sqrt{2}\ U_{k-2}
\geq k!\sqrt{2}\ 4^{k-2}
\geq k\bigg(\frac{2(k-1)}{e}\bigg)^{k-1},
\end{equation*}
and thus, for $0<X<1$, we complete the proof observing that
\begin{align*}
\big|\mathbf{\check m}^{(k)}(X)-Q_k(\log X)\big|
&\leq \frac{1}{\sqrt{X}}\bigg(k(k-1)\sqrt{2}\ U_{k-2}+k\sum_{j=1}^{k-1}\mbinom{k-1}{j}
|\gamma_{k-1-j}|A_j\bigg)\\
&=\frac{A_k}{\sqrt{X}}=\frac{A_k}{\sqrt{X}}+\frac{B_k\ \mathds{1}_{X\geq 10^{12}}(X)}{\log X}.\qedhere
\end{align*}
\end{proofbold}

Evaluating \eqref{Ak_rec} and \eqref{Bk_rec} for
$3\leq k\leq7$ gives the following upper bounds.

\begin{corollary}
Let $Q_k$ be defined by \eqref{Pol_2}.
For every $X>0$, we have
\begin{align*}
\big|\mathbf{\check m}^{(3)}(X)-Q_3(\log X)\big|
&\leq\frac{40.4500}{\sqrt{X}}
+\frac{1.9583\ \mathds{1}_{{X\geq 10^{12}}}(X)}
{10^3\log X},\\
\big|\mathbf{\check m}^{(4)}(X)-Q_4(\log X)\big|
&\leq\frac{655.3591}{\sqrt{X}}
+\frac{1.5948\ \mathds{1}_{{X\geq 10^{12}}}(X)}
{10^2\log X},\\
\big|\mathbf{\check m}^{(5)}(X)-Q_5(\log X)\big|
&\leq\frac{12933.0043}{\sqrt{X}}
+\frac{1.5838\ \mathds{1}_{{X\geq 10^{12}}}(X)}
{10\log X},\\
\big|\mathbf{\check m}^{(6)}(X)-Q_6(\log X)\big|
&\leq\frac{307902.5514}{\sqrt{X}}
+\frac{1.8966\ \mathds{1}_{{X\geq 10^{12}}}(X)}
{\log X},\\
\big|\mathbf{\check m}^{(7)}(X)-Q_7(\log X)\big|
&\leq\frac{8592314.3066}{\sqrt{X}}
+\frac{26.7712\ \mathds{1}_{{X\geq 10^{12}}}(X)}
{\log X}.
\end{align*}
\end{corollary}

\begin{remark}
The coefficients of $1/\sqrt{X}$ in the preceding bounds
grow rapidly with $k$. At $X=10^{12}$, these terms already
exceed $0.3$ for $k=6$ and $8.5$ for $k=7$.
\end{remark}

%%%%%%%%%%%%%%%%%%%%%%%%%%%%%%%%%%%%%%%%%%%%%%%%%%%%%%%%%%%%%%
\section{Coprimality}\label{coprimality}
%%%%%%%%%%%%%%%%%%%%%%%%%%%%%%%%%%%%%%%%%%%%%%%%%%%%%%%%%%%%%%

 Consider now $q,d\in\mathbb{Z}_{\geq 1}$. We write $d\mid q^\infty$ if every prime divisor of $d$
divides $q$. 
With this definition, and by using the following identity, given for example in \cite[Lemma 2]{Graham*78} and \cite[Eq. (5.73)]{Helfgott19}, one can study a sum with coprimality conditions from the same sum without such conditions.  

\begin{lemma}\label{Idd}
 We have the identity 
 \begin{equation*}
 \sum_{d|q^{\infty},d|n}\mu\bigg(\frac{n}{d}\bigg)=\mu(n)\mathds{1}_{\{(n,q)=1\}}(n).
 \end{equation*}
Hence, for any function $h:\mathbb{Z}_{\geq 1}\to\mathbb{C}$, we have the formal identity
 \begin{equation*}
 \sum_{\substack{n\\(n,q)=1}}\frac{\mu(n)}{n}h(n)=\sum_{d|q^\infty}\frac{1}{d}\sum_{n}\frac{\mu(n)}{n}h(dn).
 \end{equation*}
 \end{lemma} 

 Lemma \ref{Idd} is meaningful since, as O. Ramar\'e points out in \cite[\S 1]{RA13}, on using merely a M\"obius inversion $\sum_{d|q,d|n}\mu(d)=\mathds{1}_{\{(n,q)=1\}}(n)$ in Lemma \ref{Idd}, one would be led back to a sum with coprimality conditions. In the applications below, the sums are finite or absolutely convergent.
 
\begin{lemma}\label{aux}
Let $q\in\mathbb{Z}_{\geq 1}$, $X>0$, and $s\in\mathbb{C}$
with $\Re s>0$.
\begin{enumerate}[label=(\alph*),ref=\alph*]
\item\label{aux_a} We have
\begin{equation*}
\sum_{d|q^\infty}\frac{1}{d^s}
=\prod_{p|q}\bigg(1-\frac{1}{p^s}\bigg)^{-1}=\frac{q^s}{\varphi_s(q)}.
\end{equation*}

\item\label{aux_b} For every integer $k\geq 0$, we have
\begin{equation*}
\sum_{d|q^\infty}\frac{1}{d^s}
\log^k\bigg(\frac{X}{d}\bigg)
=
\sum_{j=0}^{k}\mbinom{k}{j}\log^{k-j}X\ 
\frac{\mathrm{d}^{j}}{\mathrm{d}s^{j}}\bigg[\frac{q^s}{\varphi_s(q)}\bigg].
\end{equation*}
\end{enumerate}
\end{lemma}

\begin{proofbold}\eqref{aux_a}
Every divisor of $q^\infty$ has a unique representation $d=\prod_{p|q}p^{\alpha_p}$ with $\alpha_p\geq 0$.
Since $\Re s>0$, each corresponding geometric series converges normally. Hence
\begin{equation*}
\sum_{d|q^\infty}\frac{1}{d^s}
=\prod_{p|q}\sum_{j=0}^{\infty}\frac{1}{p^{js}}
=\prod_{p|q}\bigg(1-\frac{1}{p^s}\bigg)^{-1}
=\frac{q^s}{\varphi_s(q)}.
\end{equation*}
\noindent\eqref{aux_b} For every integer $j\geq 0$, the series $
\sum_{d|q^\infty}\frac{\log^jd}{d^s}$
converges normally in
$\Re s>0$. Thus, differentiating the identity in
(a) term by term gives
\begin{equation}\label{diff}
\frac{\mathrm{d}^{j}}{\mathrm{d}s^{j}}\bigg[\frac{q^s}{\varphi_s(q)}\bigg]
=(-1)^j\sum_{d|q^\infty}\frac{\log^jd}{d^s}.
\end{equation}
Now, the binomial theorem yields
\begin{equation*}
\sum_{d|q^\infty}\frac{\log^k(X/d)}{d^s}=\sum_{j=0}^{k}\mbinom{k}{j}
\log^{k-j}X(-1)^j\sum_{d|q^\infty}\frac{\log^jd}{d^s}
\end{equation*}
By using \eqref{diff}, we derive the result.
\end{proofbold}

\subsection{A recurrence on derivatives}
\begin{lemma}\label{derivatives_phi}
Let $q,j\in\mathbb{Z}_{\geq 1}$. Then
\begin{align*}
\frac{\mathrm{d}^{j}}{\mathrm{d}s^{j}}
\bigg[\frac{q^s}{\varphi_s(q)}
\bigg|_{s=1}
=\phantom{.}&\sum_{i=0}^{j-1}\mbinom{j-1}{i}(-1)^{i+1}
\frac{\mathrm{d}^{j-1-i}}{\mathrm{d}s^{j-1-i}}
\bigg[\frac{q^s}{\varphi_s(q)}
\bigg|_{s=1}\cdot\sum_{p|q}\log^{i+1}p\ 
\frac{\mathrm{d}^{i}}{\mathrm{d}z^{i}}
\bigg[\frac{1}{p/e^{z}-1}
\bigg|_{z=0}.
\end{align*}
\end{lemma}

\begin{proofbold}
Since $\Re s>0$, for every $i\geq 0$, we have
\begin{equation*}
\frac{\mathrm{d}^{i+1}}{\mathrm{d}s^{i+1}}
\bigg[\sum_{p|q}\sum_{\ell=1}^{\infty}
\frac{1}{\ell p^{\ell s}}\bigg]
=(-1)^{i+1}\sum_{p|q}\log^{i+1}p\
\sum_{\ell=1}^{\infty}\frac{\ell^i}{p^{\ell s}},
\end{equation*}
where term by term differentiation is justified by normal convergence.

By applying the above identity with $i=0$, we derive
\begin{align*}
\frac{\mathrm{d}}{\mathrm{d}s}
\bigg[\frac{q^s}{\varphi_s(q)}\bigg]&=\frac{\mathrm{d}}{\mathrm{d}s}\bigg[\exp\bigg(-\sum_{p|q}\log\bigg(1-\frac{1}{p^s}\bigg)\bigg)\bigg]\\
&=\frac{\mathrm{d}}{\mathrm{d}s}\bigg[\exp\bigg(\sum_{p|q}\sum_{\ell=1}^{\infty}
\frac{1}{\ell p^{\ell s}}\bigg)\bigg]
=-\frac{q^s}{\varphi_s(q)}
\sum_{p|q}\log p\sum_{\ell=1}^{\infty}
\frac{1}{p^{\ell s}}.
\end{align*}
Applying the Leibniz differentiation rule a further
$j-1$ times in the preceding identity and evaluating at $s=1$, we obtain
\begin{align}\label{diff_Z}
\frac{\mathrm{d}^{j}}{\mathrm{d}s^{j}}
\bigg[\frac{q^s}{\varphi_s(q)}\bigg|_{s=1}
&=\sum_{i=0}^{j-1}\mbinom{j-1}{i}(-1)^{i+1}
\frac{\mathrm{d}^{j-1-i}}{\mathrm{d}s^{j-1-i}}
\bigg[\frac{q^s}{\varphi_s(q)}\bigg|_{s=1}
\cdot\sum_{p|q}\log^{i+1}p
\sum_{\ell=1}^{\infty}\frac{\ell^i}{p^\ell}.
\end{align}
Finally, for $|e^z/p|<1$, we have the series
\begin{equation*}
\frac{1}{p/e^z-1}
=\sum_{\ell=1}^{\infty}\frac{e^{\ell z}}{p^\ell},
\end{equation*}
which converges normally on $\Re z<\log p$. Differentiating $i$ times
and evaluating at $z=0$, we obtain
\begin{equation*}
\frac{\mathrm{d}^{i}}{\mathrm{d}z^{i}}
\bigg[\frac{1}{p/e^z-1}\bigg|_{z=0}
=\sum_{\ell=1}^{\infty}\frac{\ell^i}{p^\ell}.
\end{equation*}
Substituting this identity into \eqref{diff_Z}
proves the result.
\end{proofbold}

Evaluating at $s=1$ and applying Lemma \ref{derivatives_phi} for $1\leq j\leq4$ gives
\begin{align*}
\frac{\varphi(q)}{q}
\bigg[\frac{q^s}{\varphi_s(q)}\bigg|_{s=1}
=&\ 1,\qquad\qquad\frac{\varphi(q)}{q}
\frac{\mathrm{d}}{\mathrm{d}s}
\bigg[\frac{q^s}{\varphi_s(q)}\bigg|_{s=1}
=-
\sum_{p|q}\frac{\log p}{p-1},\\
\frac{\varphi(q)}{q}
\frac{\mathrm{d}^{2}}{\mathrm{d}s^{2}}
\bigg[\frac{q^s}{\varphi_s(q)}\bigg|_{s=1}
=&
\ \bigg(\sum_{p|q}\frac{\log p}{p-1}\bigg)^2
+\sum_{p|q}\frac{p\log^2p }{(p-1)^2},\\
\frac{\varphi(q)}{q}
\frac{\mathrm{d}^{3}}{\mathrm{d}s^{3}}
\bigg[\frac{q^s}{\varphi_s(q)}\bigg|_{s=1}
=&-
\bigg(\sum_{p|q}\frac{\log p}{p-1}\bigg)^3
-3\sum_{p|q}\frac{\log p}{p-1}\cdot
\sum_{p|q}\frac{p\log^2p }{(p-1)^2}
-\sum_{p|q}\frac{p(p+1)\log^3p }{(p-1)^3},\\
\frac{\varphi(q)}{q}
\frac{\mathrm{d}^{4}}{\mathrm{d}s^{4}}
\bigg[\frac{q^s}{\varphi_s(q)}\bigg|_{s=1}
=&\ \bigg(\sum_{p|q}\frac{\log p}{p-1}\bigg)^4
+6\bigg(\sum_{p|q}\frac{\log p}{p-1}\bigg)^2
\cdot\sum_{p|q}\frac{p\log^2p }{(p-1)^2}
+3\bigg(\sum_{p|q}\frac{p\log^2p }{(p-1)^2}\bigg)^2\\
&\quad+4\sum_{p|q}\frac{\log p}{p-1}\cdot
\sum_{p|q}\frac{p(p+1)\log^3p }{(p-1)^3}
+\sum_{p|q}\frac{p(p^2+4p+1)\log^4p }{(p-1)^4}.
\end{align*}

 \subsection{Theorem \ref{general_polynomial} with coprimality conditions}
 
 Let $Q_k$ be as in \eqref{Pol_2}. For $q\in\mathbb{Z}_{\geq 1}$, define $Q_{0,q}=0$ and
\begin{equation}\label{Qkq}
Q_{k,q}(t)
=\sum_{j=0}^{k-1}\mbinom{k}{j}Q_{k-j}(t)
\frac{\mathrm{d}^{j}}{\mathrm{d}s^{j}}
\bigg[\frac{q^s}{\varphi_s(q)}\bigg|_{s=1},\qquad k\geq 1.
\end{equation}
We have $Q_{k,1}=Q_k$. Moreover, for $k\geq 1$ the polynomial $Q_{k,q}$ has degree
$k-1$ and leading coefficient $\frac{kq}{\varphi(q)}$. Indeed, the term $j=0$ in \eqref{Qkq} is
$\frac{q}{\varphi(q)}Q_k(t)$, of degree $k-1$ and with leading coefficient $\frac{kq}{\varphi(q)}$, while every other term
has degree at most $k-2$.  

 \begin{theorem}\label{coprime_polynomial}
Let $q\in\mathbb{Z}_{\geq 1}$, $k\in\mathbb{Z}_{\geq 0}$ and $A_k,B_k$ be as in Theorem
\ref{general_polynomial}, and put
$\theta=1-\frac{1}{12\log 10}$.
Then, for every $X>0$,
\begin{equation}\label{coprime_polynomial_bound}
\big|\mathbf{\check m}_q^{(k)}(X)-Q_{k,q}(\log X)\big|
\leq
\frac{\sqrt{q}}{\varphi_{1/2}(q)}
\frac{A_k}{\sqrt{X}}
+\frac{q^\theta}{\varphi_\theta(q)}
\frac{B_k\ \mathds{1}_{X\geq 10^{12}}(X)}{\log X}.
\end{equation}
\end{theorem}

\begin{proofbold}
By Lemma \ref{Idd}, for every $X>0$,
\begin{equation}\label{awe}
\mathbf{\check m}_q^{(k)}(X)
=\sum_{d|q^\infty}\frac{1}{d}
\mathbf{\check m}^{(k)}\bigg(\frac{X}{d}\bigg).
\end{equation}
For $k\geq 1$, since $Q'_k(t)=kQ_{k-1}(t)$, Taylor's formula gives
\begin{equation}\label{Taylor}
Q_k(t-\log d)
=\sum_{j=0}^{k-1}\mbinom{k}{j}
Q_{k-j}(t)(-1)^j\log^jd.
\end{equation}
Subsequently, by the identity \eqref{diff}, we obtain
\begin{align}\label{Qkq_alt}
\sum_{d|q^\infty}\frac{1}{d}Q_k(t-\log d)
&=\sum_{j=0}^{k-1}\mbinom{k}{j}Q_{k-j}(t)
(-1)^j\sum_{d|q^\infty}\frac{\log^jd}{d}\nonumber\\
&\qquad=\sum_{j=0}^{k-1}\mbinom{k}{j}Q_{k-j}(t)
\frac{\mathrm{d}^{j}}{\mathrm{d}s^{j}}
\bigg[\frac{q^s}{\varphi_s(q)}\bigg|_{s=1}
=Q_{k,q}(t).
\end{align}
The interchange of summations above is justified by normal convergence. Moreover, the same identity holds for $k=0$, since $Q_0=Q_{0,q}=0$. 

On the other hand, by using identity \eqref{awe} and then identity \eqref{Qkq_alt} with $t=\log X$ we derive
\begin{align*}
\big|\mathbf{\check m}_q^{(k)}(X)-Q_{k,q}(\log X)\big|&=
\bigg|\sum_{d|q^\infty}\frac{1}{d}\bigg(\mathbf{\check m}^{(k)}\bigg(\frac{X}{d}\bigg)-Q_k\bigg(\log\bigg(\frac{X}{d}\bigg)\bigg)\bigg)\bigg|\\
&\leq\sum_{d|q^\infty}\frac{1}{d}\bigg|\mathbf{\check m}^{(k)}\bigg(\frac{X}{d}\bigg)-Q_k\bigg(\log\bigg(\frac{X}{d}\bigg)\bigg)\bigg|.
\end{align*}
Thus, thanks to Theorem \ref{general_polynomial}, we conclude
\begin{equation}\label{marvel}
\big|\mathbf{\check m}_q^{(k)}(X)-Q_{k,q}(\log X)\big|
\leq\frac{A_k}{\sqrt{X}}
\sum_{d|q^\infty}\frac{1}{\sqrt{d}}+B_k\mathds{1}_{X\geq 10^{12}}(X)
\sum_{\substack{d|q^\infty\\d\leq X/10^{12}}}
\frac{1}{d\log(X/d)}.
\end{equation}
Next, for $X\geq 10^{12}$, the function
$t\mapsto\frac{t^{-(1-\theta)}}{\log(X/t)}$
is decreasing on $[1,X/10^{12}]$, where $\theta=1-\frac{1}{12\log 10}$. Consequently,
\begin{equation}\label{calculus}
\frac{1}{d\log(X/d)}
\leq\frac{1}{d^\theta\log X},
\qquad\text{ for }1\leq d\leq X/10^{12},
\end{equation}
and thus Lemma \ref{aux}\eqref{aux_a} gives
\begin{equation*}
\sum_{\substack{d|q^\infty\\d\leq X/10^{12}}}
\frac{1}{d\log(X/d)}
\leq\frac{q^\theta}{\varphi_\theta(q)\log X},
\qquad
\sum_{d|q^\infty}\frac{1}{\sqrt{d}}
=\frac{\sqrt{q}}{\varphi_{1/2}(q)},
\end{equation*}
which we use in the estimate \eqref{marvel} to obtain \eqref{coprime_polynomial_bound}.
\end{proofbold}

By using the definition \eqref{Qkq} and
Lemma \ref{derivatives_phi} with derivatives
of order $1$ and $2$, the first four generalised polynomials $Q_{k,q}$ are
\begin{align*}
Q_{0,q}(t)&=0,\qquad Q_{1,q}(t)=\frac{q}{\varphi(q)},\qquad Q_{2,q}(t)=\frac{2q}{\varphi(q)}
\bigg(t-\gamma-\sum_{p|q}\frac{\log p}{p-1}\bigg),\\
Q&_{3,q}(t)=\frac{3q}{\varphi(q)}
\bigg[\bigg(t-\gamma-\sum_{p|q}\frac{\log p}{p-1}\bigg)^2+\gamma^2+2\gamma_1+\sum_{p|q}\frac{p\log^2p}{(p-1)^2}\bigg].
\end{align*}

\begin{corollary}
Let $q\in\mathbb{Z}_{\geq 1}$, $X>0$ and
$\theta=1-\frac{1}{12\log 10}$. 
Then
\begin{align*}
|m_q(X)|=|m_q(X)-Q_{0,q}(\log X)|
&\leq
\frac{\sqrt{q}}{\varphi_{1/2}(q)}
\frac{\sqrt{2}}{\sqrt{X}}
+\frac{q^\theta}{\varphi_\theta(q)}
\frac{1.862\ \mathds{1}_{X\geq 10^{12}}(X)}
{10^4\log X},\\
\big|\check m_q(X)-Q_{1,q}(\log X)\big|
&\leq
\frac{\sqrt{q}}{\varphi_{1/2}(q)}
\frac{1}{\sqrt{X}}
+\frac{q^\theta}{\varphi_\theta(q)}
\frac{8.55\ \mathds{1}_{X\geq 10^{12}}(X)}
{10^6\log X},\\
\big|\check{\check m}_q(X)-Q_{2,q}(\log X)\big|
&\leq
\frac{\sqrt{q}}{\varphi_{1/2}(q)}
\frac{4e^{\frac{\gamma}{2}-1}}{\sqrt{X}}
+\frac{q^\theta}{\varphi_\theta(q)}
\frac{2.989\ \mathds{1}_{X\geq 10^{12}}(X)}
{10^4\log X},\\
\big|\mathbf{\check m}_q^{(3)}(X)-Q_{3,q}(\log X)\big|
&\leq
\frac{\sqrt{q}}{\varphi_{1/2}(q)}
\frac{40.4500}{\sqrt{X}}
+\frac{q^\theta}{\varphi_\theta(q)}
\frac{1.9583\ \mathds{1}_{X\geq 10^{12}}(X)}
{10^3\log X}.
\end{align*}
\end{corollary}

\begin{remark} The first estimate in the preceding corollary also holds
with $(\theta,\mathds{1}_{X\geq10^{12}}(X))$ replaced by
$(\theta',\mathds{1}_{X\geq10^{14}}(X))$, where
$\theta'=1-\frac{1}{14\log 10}$, since the finite range bound for $m$ used in Lemma 5.1
is valid up to $10^{14}$ and the proof of Theorem 7.4 also applies with this threshold.

With this modification to the first estimate, the first three estimates of the corollary sharpen the bounds for $m_q(X)$ and
$\check m_q(X)$ in \cite[Prop. 5.15]{Helfgott19}
and for $\check{\check m}_q(X)$ in \cite[Lemma 3.2]{SZA*22}, respectively.
\end{remark}

%%%%%%%%%%%%%%%%%%%%%%%%%%%%%%%%%%%%%%%%%%%%%%%%%%%%%%%%%%%%%%
\section{The variant $\mathbf{\tilde m}_q^{(k)}(X)$}\label{tilde}
%%%%%%%%%%%%%%%%%%%%%%%%%%%%%%%%%%%%%%%%%%%%%%%%%%%%%%%%%%%%%%

We now transfer the estimates of Theorem
\ref{general_polynomial} to the functions
$\mathbf{\tilde m}_q^{(k)}$ defined in \eqref{definitions}.
The argument is analogous to the proof of \cite[Lemma 3.5]{SZA*22}. As in \S\ref{coprimality}, a convolution with nonnegative coefficients gives both the main polynomial and the error bound.

\subsection{Some identities}

For $q\in\mathbb{Z}_{\geq 1}$ define
\begin{equation*}
w_q(n)=\prod_{\substack{p|n\\p\nmid q}}\frac{1}{p+1}.
\end{equation*}

\begin{lemma}\label{q:id}
Let $n,q,d\in\mathbb{Z}_{\geq 1}$. Then
\begin{equation*}
d|n,\quad (d,q)=1,\quad \mu^2(d)=1,\quad
\frac{n}{d}|(dq)^\infty
\quad\Longleftrightarrow\quad
d=\prod_{\substack{p|n\\p\nmid q}}p.
\end{equation*}
\end{lemma}

\begin{proofbold}[Proof]
[$\ \Longrightarrow\ $]
Since $\frac{n}{d}|(qd)^{\infty}$, we have $n|(dq)^{\infty}$. Since $(d,q)=1$, a prime divides $d$ if and only if it divides $n$ and does not divide $q$. We conclude by squarefreeness. 

[$\ \Longleftarrow\ $] If $d=\prod_{\substack{p|n\\p\nmid q}}p$ then $d|n$, $(d,q)=1$ and $\mu^2(d)=1$. Moreover,
\begin{equation*}
\frac{n}{d}
=\prod_{\substack{p^{\alpha}||n\\p|q}}p^{\alpha}\prod_{\substack{p^{\alpha}||n\\p\nmid q}}p^{\alpha}\prod_{\substack{p|n\\p\nmid q}}\frac{1}{p}
=\prod_{\substack{p^{\alpha}||n\\p|q}}p^{\alpha}\prod_{\substack{p^{\alpha}||n\\p\nmid q}}p^{\alpha-1}.
\end{equation*}
Thus, if $p|\frac{n}{d}$ and $p\nmid q$, then $p|d$, whence $\frac{n}{d}|(dq)^{\infty}$.
\end{proofbold}

\begin{lemma}\label{tilde_convolution}
Let $k\in\mathbb{Z}_{\geq 0}$ and $q\in\mathbb{Z}_{\geq 1}$. For every $X>0$, we have
\begin{equation*}
\mathbf{\tilde m}_q^{(k)}(X)=\sum_{n=1}^{\infty}\frac{w_q(n)}{n}
\mathbf{\check m}^{(k)}
\bigg(\frac{X}{n}\bigg).
\end{equation*}
\end{lemma}

\begin{proofbold}
Since $\sigma$ is the sum-of-divisors function, we have
\begin{equation*}
\frac{\mu^2(n)n}{\sigma(n)}
=\mu^2(n)\prod_{p|n}\bigg(1-\frac{1}{p+1}\bigg)
=\mu^2(n)\sum_{d|n}\frac{\mu(d)}{\sigma(d)}.
\end{equation*}
Hence, by recalling definition \eqref{definitions}, we may use the above identity as follows
\begin{align*}
\mathbf{\tilde{m}}_q^{(k)}(X)=\sum_{\substack{n\leq X\\(n,q)=1}}\frac{\mu(n)}{\sigma(n)}\log^k\bigg(\frac{X}{n}\bigg)
&=\sum_{\substack{de\leq X\\(de,q)=1\\(d,e)=1}}
\frac{\mu^2(d)\mu(e)}{d\sigma(d)e}
\log^k\bigg(\frac{X}{de}\bigg)
=\sum_{\substack{d\\(d,q)=1}}
\frac{\mu^2(d)}{d\sigma(d)}
\mathbf{\check m}_{dq}^{(k)}
\bigg(\frac{X}{d}\bigg),
\end{align*}
where the above condition $\{d\leq X\}$ is encoded by the definition of $\mathbf{\check m}_{dq}^{(k)}(X/d)$. Now, by applying Lemma \ref{Idd} to
$\mathbf{\check m}_{dq}^{(k)}(X/d)$ gives
\begin{equation*}
\mathbf{\tilde m}_q^{(k)}(X)=\sum_{\substack{d\\(d,q)=1}}
\frac{\mu^2(d)}{d\sigma(d)}
\sum_{\ell|(dq)^\infty}\frac{1}{\ell}
\mathbf{\check m}^{(k)}\bigg(\frac{X}{d\ell}\bigg)
=\sum_{n\leq X}\frac{1}{n}
\mathbf{\check m}^{(k)}\bigg(\frac{X}{n}\bigg)
\sum_{\substack{d|n\\(d,q)=1\\\frac{n}{d}|(dq)^\infty}}
\frac{\mu^2(d)}{\sigma(d)}.
\end{equation*}
The result is concluded with the help of Lemma \ref{q:id} since
\begin{equation*}
\sum_{\substack{d|n\\(d,q)=1\\\frac{n}{d}|(dq)^\infty}}
\frac{\mu^2(d)}{\sigma(d)}
=\sum_{d=\prod_{p|n,p\nmid q}p}\frac{1}{\sigma(d)}
=\prod_{\substack{p|n\\p\nmid q}}\frac{1}{p+1}
=w_q(n).\qedhere
\end{equation*}
\end{proofbold}

\begin{proofbold}[\textbf{Proof of Theorem \ref{positive}}]
Let $k,q\in\mathbb Z_{\geq 1}$ and $X>0$. By \cite[Thm.~1.2]{RamareZuniga-Alterman*26} with $s=1$ and $q=1$ we have, for any $n\leq X$,
 $\mathbf{\check{m}}^{(k)}\big(\frac{X}{n}\big)\geq 0$. Thus, since $w_q(n)\geq 0$, Lemma \ref{tilde_convolution} gives
 \begin{equation*}
\mathbf{\tilde m}_q^{(k)}(X)=\sum_{n=1}^{\infty}\frac{w_q(n)}{n}
\mathbf{\check m}^{(k)}\bigg(\frac{X}{n}\bigg)=\sum_{n\leq X}\frac{w_q(n)}{n}
\mathbf{\check m}^{(k)}\bigg(\frac{X}{n}\bigg)\geq 0.\qedhere
\end{equation*}
 \end{proofbold}

\subsection{General estimates}

\begin{lemma}\label{Hq_s}
Let $q\in\mathbb{Z}_{\geq 1}$. For $\Re s>0$, define 
$H_q(s)
:=\sum_{n=1}^{\infty}\frac{w_q(n)}{n^s}$. Then
\begin{equation}\label{tilde_euler_product}
H_q(s)
=\frac{q^s}{\varphi_s(q)}
\prod_{p\nmid q}
\bigg(1+\frac{1}{(p+1)(p^s-1)}\bigg)
=\mathpzc p_s(q)
\prod_p\bigg(1+\frac{1}{(p+1)(p^s-1)}\bigg),
\end{equation}
where $\mathpzc p_s(q)$ is defined in \eqref{functions}.
The series and product defining $H_q$ converge normally in the
half-plane $\Re s>0$. In particular, we have
\begin{equation}\label{H_q(1)}
H_q(1)=\zeta(2)\mathpzc p_1(q).
\end{equation}

\end{lemma}

\begin{proofbold}
The function $w_q$ is multiplicative and for $\alpha\geq 1$, we have 
\begin{equation*}
\frac{w_q(p^\alpha)}{p^{\alpha s}}=\frac{1}{p^{\alpha s}}\quad\text{ if }\ p|q\qquad\text{ and }\qquad \frac{w_q(p^\alpha)}{p^{\alpha s}}=\frac{1}{p^{\alpha s}(p+1)},\quad\text{ if }\ p\nmid q.
\end{equation*}
Since there is only a finite number of primes that divide $q$, we conclude from the second expression above that the series defining $H_q(s)$ converges normally in $\{\Re s>0\}$. Further, in that region, we have the following Euler product
\begin{equation*}
H_q(s)=\prod_{p|q}\bigg(\sum_{\alpha\geq 0}\frac{1}{p^{\alpha s}}\bigg)\prod_{p\nmid q}\bigg(1+\frac{1}{(p+1)}\sum_{\alpha\geq 1}\frac{1}{p^{\alpha s}}\bigg)
=\frac{q^s}{\varphi_s(q)}
\prod_{p\nmid q}
\bigg(1+\frac{1}{(p+1)(p^s-1)}\bigg).
\end{equation*}
Rearranging the Euler product gives \eqref{tilde_euler_product}. Finally, at $s=1$, we have
\begin{equation*}
H_q(1)
=\frac{q}{\varphi(q)}
\prod_{p\nmid q}\bigg(1-\frac{1}{p^2}\bigg)^{-1}
=\zeta(2)\mathpzc p_1(q).\qedhere
\end{equation*}
\end{proofbold}
Let $Q_k$ be as in \eqref{Pol_2}. For $q\in\mathbb{Z}_{\geq 1}$,
set $\widetilde Q_{0,q}=0$ and for $k\in\mathbb{Z}_{\geq 1}$ define
\begin{equation}\label{tilde_Pol}
\widetilde Q_{k,q}(t)
:=\sum_{j=0}^{k-1}\mbinom{k}{j}Q_{k-j}(t)
\frac{\mathrm{d}^{j}}{\mathrm{d}s^{j}}
\bigg[H_q(s)\bigg|_{s=1}.
\end{equation}
Observe that, for $k\geq 1$, the polynomial $\widetilde Q_{k,q}$
has degree $k-1$ and leading coefficient
$k\zeta(2)\mathpzc p_1(q)$. Indeed, the term $j=0$ in \eqref{tilde_Pol}
is $H_q(1)Q_k(t)$ of degree $k-1$ and leading coefficient $H_q(1)k$, whose value is given by \eqref{H_q(1)}, and the remaining terms have smaller degree. 

\begin{lemma}\label{tilde_initial_bounds}
Let $q\in\mathbb{Z}_{\geq 1}$ and set
$\theta=1-\frac{1}{12\log10}$.
For $X>0$, we have
\begin{equation}
\begin{aligned}
\big|\mathbf{\tilde m}_q^{(0)}(X)\big|
&\leq
\mathpzc p_{1/2}(q)\frac{\sqrt{2}}{\sqrt X}
+\mathpzc p_\theta(q)\frac{4.159}{10^4}
 \frac{\mathds{1}_{X\geq 10^{12}}(X)}{\log X},
\\
\big|\mathbf{\tilde m}_q^{(1)}(X)
      -\zeta(2)\mathpzc p_1(q)\big|
&\leq
\mathpzc p_{1/2}(q)\frac{\zeta(2)}{\sqrt X}
+\mathpzc p_\theta(q)\frac{6.781}{10^5}
 \frac{\mathds{1}_{X\geq 10^{12}}(X)}{\log X}.
\end{aligned}
\end{equation}
\end{lemma}

\begin{proofbold}
By \cite[Eq. (6.132)]{Helfgott19}
and \cite[Lemma 6.3]{Helfgott19} with $v=1$,
respectively,  we have
\begin{equation*}
\big|\mathbf{\tilde m}^{(0)}(X)\big|
\leq\frac{\sqrt{2}}{\sqrt{X}},\quad
\big|\mathbf{\tilde m}^{(1)}(X)-\zeta(2)\big|
\leq\frac{\zeta(2)}{\sqrt{X}},\quad 0<X\leq10^{12}.
\end{equation*}
By \eqref{H_q(1)}, we have the identity
\begin{equation*}
\mathbf{\tilde m}^{(1)}(X)-\zeta(2)
=\mathbf{\tilde m}^{(1)}(X)-H_1(1)
=
\sum_{n=1}^{\infty}\frac{w_1(n)}{n}
\bigg(\check m\bigg(\frac Xn\bigg)-1\bigg).
\end{equation*}
By applying Lemmas \ref{tilde_convolution} and \ref{initial_estimates}, together with
\eqref{calculus} for the logarithmic terms, we have
\begin{align}\label{valss}
\big|\mathbf{\tilde m}^{(0)}(X)\big|
&\leq
\frac{H_1(1/2)\sqrt{2}}{\sqrt X}
+
\frac{1.862\ H_1(\theta) \mathds{1}_{X\geq 10^{12}}(X)}{10^4\log X},\nonumber\\
\big|\mathbf{\tilde m}^{(1)}(X)-\zeta(2)\big|
&\leq
\frac{H_1(1/2)}{\sqrt X}
+\frac{8.55\ H_1(\theta) \mathds{1}_{X\geq 10^{12}}(X)}{10^6\log X}.
\end{align}
Moreover, thanks to the proof of \cite[Lemma~3.5]{SZA*22} and Lemma \ref{Hq_s}, we have 
\begin{equation*}
\sum_{n=1}^{\infty}\frac{w_1(n)}{\sqrt n}=H_1(1/2)
\leq 3.575,
\qquad
\sum_{n=1}^{\infty}\frac{w_1(n)}{n^\theta}=H_1(\theta)
\leq 1.693.
\end{equation*}
Furthermore, since the function $t\geq e^2\mapsto\log t/\sqrt{t}$ is decreasing, we have from \eqref{valss} and for $X>10^{12}$,
\begin{align*}
\big|\mathbf{\tilde m}^{(0)}(X)\big|
&\leq\frac{\sqrt{2}}{\sqrt X}+\frac{1}{\log X}\bigg(\frac{1.693\cdot1.862}{10^4}+\frac{(3.575-1)\sqrt{2}\cdot 12\log10}{10^6}\bigg),\nonumber\\
\big|\mathbf{\tilde m}^{(1)}(X)-\zeta(2)\big|
&\leq\frac{\zeta(2)}{\sqrt X}+\frac{1}{\log X}\bigg(\frac{1.693\cdot8.55}{10^6}+\frac{(3.575-\zeta(2))\,12\log10}{10^6}\bigg).
\end{align*}
Hence, for $X>0$, by  upper bounding the above constants accompanying the logarithmic terms, we obtain the result for $q=1$.

Next, we may study coprimality conditions. For $q\in\mathbb{Z}_{\geq 1}$, let $v_q$ be the completely
multiplicative function defined on primes by
$
v_q(p)=
\frac{\mathds{1}_{p|q}}{p+1}.
$
For any $q\in\mathbb{Z}_{\geq 1}$ and $k\in\{0,1\}$, the convolution identities in
\cite[Eq.~(6.22) and the proof of Lemma~6.28]{Helfgott19} give
\begin{equation}\label{_starstar_}
\mathbf{\tilde m}_q^{(k)}(X)
-\widetilde Q_{k,q}(\log X)=
\sum_{d\mid q^\infty}v_q(d)\bigg(\mathbf{\tilde m}^{(k)}\bigg(\frac{X}{d}\bigg)
-\widetilde Q_{k,1}\bigg(\log\bigg(\frac{X}{d}\bigg)\bigg)\bigg),
\end{equation}
where the subtraction of the main term for $k=1$
is justified since 
$\widetilde Q_{1,1}(t)=\zeta(2)$ and
\begin{equation*}
\sum_{d\mid q^\infty}v_q(d)
=
\prod_{p\mid q}\left(1-\frac{1}{p+1}\right)^{-1}
=
\mathpzc p_1(q).
\end{equation*}
In fact, observe that
\begin{equation*}
\sum_{d=1}^{\infty}v_q(d)d^{1-s}=
\prod_{p\mid q}\bigg(1-\frac{p^{1-s}}{p+1}\bigg)^{-1}
=\mathpzc p_s(q),\qquad \Re s>0.
\end{equation*}
Applying the bounds just proved for $q=1$ to \eqref{_starstar_},
using \eqref{calculus} for the logarithmic terms, and evaluating
the above identity at $s=1/2$ and $s=\theta$
proves the result.
\end{proofbold}

\begin{remark}
The preceding lemma improves the global estimates for $\mathbf{\tilde m}_q^{(0)}(X)$ and $\mathbf{\tilde m}_q^{(1)}(X)$ in \cite[Lemma~6.28]{Helfgott19} and \cite[Prop.~6.8]{Helfgott19}, respectively, with $\alpha_1=1/2$.
\end{remark}

Consider the following constants
\begin{align}\label{AB_tilde}
(\tilde{A}_0,&\tilde{B}_0)=\bigg(\sqrt{2},\frac{4.159}{10^4}\bigg),\quad (\tilde{A}_1,\tilde{B}_1)=\bigg(\zeta(2),\frac{6.781}{10^5}\bigg),\nonumber\\
&\quad(\tilde{A}_k,\tilde{B}_k)=(3.575\ A_k,1.693\ B_k),\quad k\geq 2,
\end{align}
where $(A_k,B_k)$, $k\geq 2$, are defined by \eqref{AB}, \eqref{Ak_rec} and \eqref{Bk_rec}. We have the following result.
\begin{theorem}\label{tilde_general_bound}
Let $q\in\mathbb{Z}_{\geq 1}$ and $k\in\mathbb{Z}_{\geq 0}$. Let $(\tilde{A}_k,\tilde{B}_k)$ be as in \eqref{AB_tilde} and put $\theta=1-\frac{1}{12\log 10}$.
Then, for every $X>0$,
\begin{equation*}
\big|\mathbf{\tilde m}_q^{(k)}(X)
-\widetilde Q_{k,q}(\log X)\big|
\leq
\mathpzc p_{1/2}(q) 
\frac{\tilde{A}_k}{\sqrt{X}}
+\mathpzc p_\theta(q) 
\frac{\tilde{B}_k\ \mathds{1}_{X\geq 10^{12}}(X)}{\log X}.
\end{equation*}
\end{theorem}

\begin{proofbold}
For $k\in\{0,1\}$, the result follows from
Lemma~\ref{tilde_initial_bounds}, since
$\widetilde Q_{0,q}(\log X)=0$ and
$\widetilde Q_{1,q}(\log X)=\zeta(2)\mathpzc p_1(q)$. 

Henceforth, assume $k\geq 2$; applying identity \eqref{Taylor} gives
\begin{align}\label{tilde_polynomial_convolution}
&\sum_{n=1}^{\infty}\frac{w_q(n)}{n}Q_k(t-\log n)=\sum_{n=1}^{\infty}\frac{w_q(n)}{n}\sum_{j=0}^{k-1}\mbinom{k}{j}
Q_{k-j}(t)(-1)^j\log^jn\\
&=\sum_{j=0}^{k-1}\mbinom{k}{j}
Q_{k-j}(t)\sum_{n=1}^{\infty}\frac{w_q(n)}{n}(-1)^j\log^jn=\sum_{j=0}^{k-1}\mbinom{k}{j}
Q_{k-j}(t)\frac{\mathrm{d}^j}{\mathrm{d}s^j}\bigg[\sum_{n=1}^{\infty}\frac{w_q(n)}{n^s}\bigg|_{s=1}=\widetilde Q_{k,q}(t).\nonumber
\end{align}
The interchange of summations and the term-by-term differentiation are justified since the series $H_q(s)$ defined in Lemma \ref{Hq_s} and all its derivatives converge normally in $\{\Re s>0\}$. 

On the other hand, by Lemma \ref{tilde_convolution}, identity
\eqref{tilde_polynomial_convolution} with $t=\log X$ and Theorem
\ref{general_polynomial}, for every $X>0$ and $k\in\mathbb{Z}_{\geq 2}$, we obtain
\begin{align*}
\big|\mathbf{\tilde m}_q^{(k)}(X)
-\widetilde Q_{k,q}(\log X)\big|
&\leq
\sum_{n=1}^{\infty}\frac{w_q(n)}{n}
\bigg|\mathbf{\check m}^{(k)}
\bigg(\frac{X}{n}\bigg)
-Q_k\bigg(\log\frac{X}{n}\bigg)\bigg|\\
&\leq
\frac{A_k}{\sqrt{X}}
\sum_{n=1}^{\infty}\frac{w_q(n)}{\sqrt{n}}
+\mathds{1}_{X\geq 10^{12}}(X)\ B_k\sum_{n\leq X/10^{12}}
\frac{w_q(n)}{n\log(X/n)}.
\end{align*}
The same analysis to obtain inequality \eqref{calculus} together with Lemma \ref{Hq_s} give
\begin{equation*}
\big|\mathbf{\tilde m}_q^{(k)}(X)
-\widetilde Q_{k,q}(\log X)\big|
\leq\frac{A_kH_q(1/2)}{\sqrt{X}}
+\frac{B_kH_q(\theta)\
\mathds{1}_{X\geq 10^{12}}(X)}{\log X}.
\end{equation*}
The product estimates in the proof of
\cite[Lemma 3.5]{SZA*22} and the expression \eqref{tilde_euler_product} yield
$H_q(1/2)\leq\mathpzc p_{1/2}(q) 3.575$ and
$H_q(\theta)\leq\mathpzc p_\theta(q) 1.693 $.
Since $\tilde A_k=3.575\ A_k$ and $\tilde B_k=1.693\ B_k$
for $k\geq 2$, this proves the stated inequality.
\end{proofbold}

\subsection{First polynomials}

Put
\begin{equation}\label{abq}
\mathfrak a_q
=\gamma+\sum_p\frac{\log p}{p(p-1)}
+\sum_{p|q}\frac{\log p}{p},\quad \mathfrak b_q
=\sum_p
\frac{(p^2+p-1)\log^2p}{p^2(p-1)^2}
+\sum_{p|q}\frac{(p+1)\log^2p}{p^2}.
\end{equation}
Note that
\begin{equation*}
\sum_p\frac{\log p}{p(p-1)}
=0.7553666108\ldots,
\quad\text{ and }\quad
\sum_p\frac{(p^2+p-1)\log^2p}{p^2(p-1)^2}
=1.6259635276\ldots.
\end{equation*}
With the notation \eqref{abq}, the first four polynomials $\widetilde Q_{k,q}(t)$ are
\begin{align}\label{tilde_first_polynomials}
\widetilde Q_{0,q}(t)&=0,\qquad
\widetilde Q_{1,q}(t)=\zeta(2)\mathpzc p_1(q),\qquad
\widetilde Q_{2,q}(t)=2\zeta(2)\mathpzc p_1(q)(t-\mathfrak a_q),
\\
&\widetilde Q_{3,q}(t)
=3\zeta(2)\mathpzc p_1(q)
\bigg((t-\mathfrak a_q)^2
+\gamma^2+2\gamma_1+\mathfrak b_q\bigg).
\nonumber
\end{align}
By definition \eqref{tilde_Pol}, it suffices to calculate
$H_q(1)$, $H_q'(1)$ and $H_q''(1)$.
The value of $H_q(1)$ is given by \eqref{H_q(1)}.
Subsequently, by performing logarithmic differentiation at $s=1$ via Lemma \ref{Hq_s}, we have
\begin{align*}
\frac{H'_q(1)}{H_q(1)}
&=-\sum_{p|q}\frac{\log p}{p-1}
-\sum_{p\nmid q}\frac{\log p}{p(p-1)}
=\gamma-\mathfrak a_q,\\
\frac{H''_q(1)}{H_q(1)}
-\bigg(\frac{H'_q(1)}{H_q(1)}\bigg)^2
&=\sum_{p|q}\frac{p\log^2p}{(p-1)^2}+\sum_{p\nmid q}
\bigg(\frac{p}{(p-1)^2}-\frac{p+1}{p^2}\bigg)
\log^2p
=\mathfrak b_q.
\end{align*}
Thus
\begin{equation*}
H'_q(1)
=H_q(1)(\gamma-\mathfrak a_q),
\qquad
H''_q(1)
=H_q(1)
\bigg((\gamma-\mathfrak a_q)^2+\mathfrak b_q\bigg).
\end{equation*}
Substituting these identities and the values of
$Q_1,Q_2,Q_3$ into \eqref{tilde_Pol} proves \eqref{tilde_first_polynomials}.

\begin{corollary}
Let $q\in\mathbb Z_{\geq1}$, $X>0$, and
$\theta=1-1/(12\log10)$. Recall definition \eqref{functions} and the values
\eqref{AB_tilde}. Then
\begin{align*}
\big|\mathbf{\tilde m}_q^{(2)}(X)
-\widetilde Q_{2,q}(\log X)\big|
&\leq
\mathpzc p_{1/2}(q)
\frac{4e^{\frac{\gamma}{2}-1}\cdot 3.575}{\sqrt{X}}
+\mathpzc p_\theta(q)
\frac{2.989\cdot 1.693 \ \mathds{1}_{X\geq 10^{12}}(X)}
{10^4\log X},\\
\big|\mathbf{\tilde m}_q^{(3)}(X)
-\widetilde Q_{3,q}(\log X)\big|
&\leq
\mathpzc p_{1/2}(q) 
\frac{40.45\cdot 3.575}{\sqrt{X}}
+\mathpzc p_\theta(q) 
\frac{1.9583\cdot 1.693 \ \mathds{1}_{X\geq 10^{12}}(X)}
{10^3\log X}.
\end{align*}
\end{corollary}
\begin{remark} The preceding corollary sharpens the estimate of $\mathbf{\tilde m}_q^{(2)}(X)$ in \cite[Lemma 3.5]{SZA*22} .
\end{remark}

\bibliographystyle{plain}
\bibliography{Official}

@article{BA12,
  author  = {Balazard, M.},
  title   = {Elementary remarks on {M}\"obius' function},
  journal = {Proc. Steklov Inst. Math.},
  volume  = {276},
  number  = {1},
  year    = {2012},
  pages   = {33--39},
  doi     = {10.1134/s008154381201004x}
}

@article{BriggsChowla*55,
  author  = {Briggs, W.E. and Chowla, S.},
  title   = {The power series coefficients of {$\zeta(s)$}},
  journal = {Amer. Math. Monthly},
  volume  = {62},
  number  = {5},
  year    = {1955},
  pages   = {323--325},
  doi     = {10.2307/2307036}
}

@article{ChirreHelfgott*25a,
  author        = {Chirre, A. and Helfgott, H.A.},
  title         = {Optimal bounds for sums of bounded arithmetic functions},
  journal      = {arXiv.2511.14736},
  year          = {2025},
  url           = {https://arxiv.org/abs/2511.14736},
  eprint        = {2511.14736},
  archiveprefix = {arXiv},
  primaryclass  = {math.NT},
  note          = {Preprint}
}

@article{Daval25,
  author  = {Daval, F.},
  title   = {Explicit conversions between summatory functions of the {M}\"{o}bius function},
  journal = {Comb. Number Theory},
  volume  = {14},
  number  = {2},
  year    = {2025},
  pages   = {163--188},
  doi     = {10.2140/cnt.2025.14.163}
}

@article{Graham*78,
  author  = {Graham, S.W.},
  title   = {An asymptotic estimate related to {S}elberg's sieve},
  journal = {J. Number Theory},
  volume  = {10},
  number  = {1},
  year    = {1978},
  pages   = {83--94},
  doi     = {10.1016/0022-314x(78)90010-0}
}

@book{Helfgott19,
  author  = {Helfgott, H.A.},
  title   = {The ternary {Goldbach} problem},
  year    = {2019},
  url     = {https://webusers.imj-prg.fr/~harald.helfgott/anglais/book.html},
  note    = {Version of December 16, 2019; accepted for publication in Annals of Mathematics Studies. https://webusers.imj-prg.fr/~harald.helfgott/anglais/book.html}
}

@article{RA13,
  author  = {Ramar\'{e}, O.},
  title   = {Explicit estimates on the summatory functions of the {M}oebius function with coprimality restrictions},
  journal = {Acta Arith.},
  volume  = {165},
  number  = {1},
  year    = {2014},
  pages   = {1--10},
  doi     = {10.4064/aa165-1-1}
}

@article{RamareZuniga-Alterman*26,
  author  = {Ramar\'e, O. and Zuniga-Alterman, S.},
  title   = {{M}\"obius function and primes: an identity factory with applications},
  journal = {Math. Comp.},
  year    = {2026},
  doi     = {10.1090/mcom/4216},
  issn    = {0025-5718},
  note    = {Published electronically on July 23, 2026}
}

@manual{SageMath*26,
  author  = {SageMath},
  title   = {Open-Source Mathematical Software System},
  year    = {2026},
  url     = {https://www.sagemath.org},
  note    = {Version 10.9}
}

@article{SZA*22,
  author  = {Zuniga-Alterman, S.},
  title   = {On a logarithmic sum related to the {Selberg} sieve},
  journal = {Acta Arith.},
  volume  = {202},
  number  = {2},
  year    = {2022},
  pages   = {123--160},
  doi     = {10.4064/aa200712-22-6}
}

\Addresses

\end{document}